\documentclass{article}

\usepackage{amsfonts}
\usepackage{amssymb}
\usepackage{amsmath}

\usepackage{hyperref}
\newcommand{\arxiv}[1]{\href{http://www.arXiv.org/abs/#1}{arXiv:#1}}
\usepackage{tikz}

\newtheorem{theorem}{Theorem}[section]

\newtheorem{lemma}[theorem]{Lemma}

\newtheorem{proposition}[theorem]{Proposition}
\newtheorem{remark}[theorem]{Remark}

\newenvironment{proof}[1][Proof]{\noindent\textbf{#1.} }{\ \rule{0.5em}{0.5em}}

\def\semr{{\mathcal S}}

\def\Nat{{\mathcal N}}
\def\supp{\operatorname{supp}}
\def\bzero{{\mathbf 0}}
\def\bunity{{\mathbf 1}}
\def\mysup{\vee}
\def\myinf{\wedge}
\def\R{{\mathbb R}}
\def\Rp{\R_+}

\def\boldunity{\bunity}
\def\digr{D}
\def\accesses{\longrightarrow}

\begin{document}

\title{$Z$-matrix equations in max algebra, nonnegative linear algebra and
other semirings\thanks{%
This research was supported by EPSRC grant RRAH12809 (all authors) and partially by RFBR-CRNF grant 11-01-93106 (S.~Sergeev)}}
\author{Peter Butkovi\v{c}\thanks{%
Corresponding author. School of Mathematics, University of Birmingham,
Edgbaston, Birmingham B15 2TT, United Kingdom, p.butkovic@bham.ac.uk}, Hans
Schneider\thanks{%
Department of Mathematics, University of Wisconsin-Madison, Madison,
Wisconsin 53706, USA, hans@math.wisc.edu} and Serge\u{\i} Sergeev\thanks{%
INRIA and CMAP Ecole Polytechnique, 91128 Palaiseau Cedex, France,
sergiej@gmail.com}}
\maketitle

\begin{abstract}
We study the max-algebraic analogue of equations involving $Z$-matrices and $%
M$-matrices, with an outlook to a more general algebraic setting. We
show that these equations can be solved using the Frobenius trace
down method in a way similar to that in non-negative linear algebra
\cite{Fro-1912,HSHershkowitz,Sch-86}, characterizing the solvability in terms of
supports and access relations.  We give a description of the
solution set as combination of the least solution and the eigenspace
of the matrix, and provide a general algebraic setting in which this result holds.

\medskip\noindent
{\bf Keywords:} Max-algebra; nonnegative linear algebra; idempotent semiring; Z-matrix equations; Kleene star

\medskip\noindent
{\bf AMS codes:} 15A80, 15A06, 15B48.
\end{abstract}

\section{Introduction}

\bigskip A $Z$-matrix is a square matrix of the form $\lambda I-A$ where $%
\lambda $ is real and $A$ is an (elementwise) nonnegative matrix. It
is called an $M$-matrix if $\lambda \geq \rho \left( A\right) ,$
where $\rho \left( A\right) $ is the Perron root (spectral radius)
of $A$ and it is nonsingular if and only if $\lambda >\rho \left(
A\right) $. Since their introduction by Ostrowski \cite{Ostrowski}
$M$-matrices have been studied in many papers and they have found
many applications. The term $Z$-matrix was introduced by
Fiedler-Pt\'{a}k \cite{fiedlerptak}.

Results on the existence, uniqueness and nonnegativity of a solution $x$ of
the equation $\left( \lambda I-A\right) x=b$ for a given nonnegative vector $%
b$ appear in many places (e.g. Berman-Plemmons \cite{Bermanplemmons}
in the case of a nonsingular $M$-matrix or an irreducible singular
$M$-matrix). Using the Frobenius normal form of $A$ and access
relation defined by the graph of the matrix, Carlson \cite{carlson}
studied the existence and uniqueness of nonnegative solutions $x$ of
this equation in the case of a
reducible singular $M$-matrix, and his results were generalized to all $Z$%
-matrices in Hershkowitz-Schneider \cite{HSHershkowitz}.

The purpose of the current paper is to prove corresponding results
in the max times algebra of nonnegative matrices, unifying and
comparing them with the results in the classical nonnegative linear
algebra. We also notice that the basic proof techniques are much
more general. In particular, we exploit a generalization of the
Frobenius trace down method~\cite{Fro-1912,Sch-86}. This
generalization is reminiscent of the universal algorithms developed by
Litvinov, Maslov, Maslova and Sobolevski\u{\i}~\cite{LM-95,LMa-00, LS-98}, based on the
earlier works on regular algebra applied to path-finding problems by
Backhouse, Carr\'e and Rote~\cite{BC-75,Rot-85}. Following this line
allows to include other examples of idempotent semirings, such as
max-min algebra~\cite{Gav:04} and distributive
lattices~\cite{Tan-98}. A more general theoretic setup is described
in Section~\ref{s:algen}. It is very close to Cohen, Gaubert, Quadrat, Singer~\cite{CGQS-05} and 
Litvinov, Maslov, Shpiz~\cite{LMS-01}.

The main object of our study is $Z$-matrix equation
\begin{equation}
Ax+b=\lambda x.  \label{Eq Ax+b=lx}
\end{equation}
over semirings. In the classical nonnegative algebra and max-plus
algebra, any $\lambda\neq 0$ is invertible, which allows to
reduce~(\ref{Eq Ax+b=lx}) to
\begin{equation}
Ax+b= x.  \label{Eq Ax+b=x}
\end{equation}
In max-plus algebra, this equation is sometimes referred to as
discrete Bellman equation, being related to the Bellman optimality
principle and dynamic programming~\cite{BCOQ,GM:08,KM:97}. In
particular, it is very well-known that this equation has the least
solution.  
However  (to the authors' knowledge) a universal and complete
description of solutions of~$\lambda x= Ax+b$ or even~$x=Ax+b$,
which would cover both classical nonnegative and max-plus algebra
cases, is not found (surprisingly) in the key monographs on max-plus algebra and
related semirings.  Such a description is what we try to achieve in
this paper, see Theorems~\ref{Th main 1} and~\ref{Th main 2}. In
brief, the results in the case of max times linear algebra are
similar to those in classical $Z$-matrix
theory~\cite{HSHershkowitz}, but they are not identical with them.
Details are given in the main sections. The situation is analogous
to that for the Perron-Frobenius equation $Ax=\lambda x$, as may be
seen by comparing the results in Butkovi\v{c}~\cite{PBbook}, Sect.
4.5, for max algebra with those in Hershkowitz and Schneider
\cite{HSHershkoverview}, Sect.~3, for classical nonnegative algebra.

The rest of the paper consists of Prerequisites (Sect.~2), Theory of
$Z$-matrix equations (Sect.~3) and Algebraic generalization
(Sect.~4). Prerequisites are devoted to the general material, mostly
about max-plus algebra: Kleene star, Frobenius normal forms and
spectral theory in the general (reducible) case. Theory of Z-matrix
equations over max-plus algebra stands on two main results.
Theorem~\ref{Th main 1} describes the solutions of~(\ref{Eq Ax+b=x})
as combinations of the least solution $A^{\ast}b$ and the
eigenvector space. We emphasize the algebraic generality of the
argument. Theorem~\ref{Th main 2} exploits the Frobenius trace down
method. This method serves both for theoretic purposes (to provide a
necessary and sufficient condition for the existence of solutions,
and to characterize the support of the least solution) and as an
algorithm for calculating it. As an outcome, we get both
combinatorial and geometric description of the solutions. The
results in max-times algebra are compared with the case of
nonnegative matrix algebra~\cite{HSHershkowitz}. The paper ends with
Section~\ref{s:algen}, devoted to an abstract algebraic setting for
which Theorem~\ref{Th main 1} holds, in the framework of semirings,
distributive lattices and lattice-ordered
groups~\cite{Bir:79,Gol:03}.

We use the conventional arithmetic notation $a+b$ and
$ab$ for the operations in semirings, to emphasize the common
features of the problem over classical nonnegative algebra and in
the idempotent semirings, viewing max-times algebra (isomorphic to
max-plus algebra) as our leading example.

We note that a complete description of solutions of~$x=Ax+b$ was also achieved by Krivulin~\cite{Kri-06,Kri-06b,Kri:09}, 
for the case of max algebra and related semirings with idempotent addition
($a\oplus a=a$). His proof of Theorem~\ref{Th main 1}, recalled here in a remark following that theorem, 
is different from the one found by the authors. We show that Krivulin's proof also
works both for max algebra and nonnegative linear algebra, and admits further algebraic generalizations.
The case of reducible matrix $A$ and general support of $b$ has been also investigated, see~\cite{Kri-06} Theorem 2
or~\cite{Kri:09} Theorem 3.2, which can be seen as a corollary of 
Theorem~\ref{Th main 2} of the present paper with application of the
max-plus spectral theory, see Theorem~\ref{t:red-spec}.

\section{Prerequisites\label{Sec Prereq}}

\subsection{Kleene star and the optimal path problem}

The main motivation of this paper is to unite and compare the
$Z$-equation theory in the classical nonnegative linear algebra, and
the max-times linear algebra. Algebraically, these structures are
semirings~\cite{Gol:03} (roughly speaking, "rings without
subtraction", see Section~\ref{s:algen} for a rigorous definition).
Thus we are focused on

{\bf Example 1: Max-times algebra.} Nonnegative numbers, endowed
with the usual multiplication $\times$ and the unusual addition
$a+b:=\max(a,b)$.

{\bf Example 2: Usual nonnegative algebra.} Nonnegative numbers,
endowed with usual arithmetics $+,\times$.

Some results in this paper will have a higher level of generality,
which we indicate by formulating them in terms of a "semiring
$\semr$". Namely, this symbol "$\semr$" applies to a more general
algebraic setting provided in Section~\ref{s:algen}, covering the
max-times algebra and the nonnegative algebra. Before reading the
last section, it may be assumed by the reader that $\semr$ means
just "max-times or usual nonnegative".

The matrix algebra over a semiring $\semr$ is defined in the usual
way, by extending the arithmetical operations to matrices and
vectors, so that $(A+B)_{ij}=a_{ij}+b_{ij}$ and $(AB)_{ik}=\sum_j
a_{ij}b_{jk}$ for matrices $A,B$ of appropriate sizes. The unit
matrix (with $1$'s on the diagonal and $0$'s off the diagonal) plays
the usual role.

Denote $N=\left\{ 1,...,n\right\} .$ For $x\in \semr^{n},$ $x>0$
means $x_{i}>0$ for every $i.$ Similarly $A>0$ for
$A\in\semr^{n\times n}.$  We also denote:
\begin{equation}
\begin{split}
\label{closures} A^{\ast } &=I+A+A^2+\ldots=\sup_{k\geq
0}(I+A+\ldots_+A^k).
\end{split}
\end{equation}
In (\ref{closures}), we have exploited the nondecreasing property of
addition. $A^{\ast}$ is also called the {\em Kleene star}, and it is
related to the optimal path problem in the following way.

The \textit{digraph associated with} $A=\left( a_{ij}\right) \in
\semr^{n\times n}$ is $D_{A}=(N,E),$ where $E=\left\{ \left(
i,j\right) ;a_{ij}>0\right\}$. The weight of a path on $D_A$ is
defined as the product of the weights of the arcs, i.e., the
corresponding matrix entries. It is easy to check (using the
distributivity law) that $(A^k)_{ij}$ is the sum of the weights of
all paths of length $k$ connecting $i$ to $j$. Further, an entry
$(A^*)_{ij}$ collects in a common summation (possibly divergent and
formal) all weights of the paths connecting $i$ to $j$, when $i\neq
j$.

Note that $A^*=(I-A)^{-1}$ in the case of the classical arithmetics,
and $A^*$ solves the optimal path problem in the case of the
max-times algebra (because the summation is maximum).

Thus the Kleene star can be described in terms of paths or access
relations in $D_A$. For $i\neq j$, we say that $i$ {\em accesses}
$j$, denoted $i\accesses j$, if there is a path of nonzero weight
connecting $i$ to $j$, equivalently, $(A^*)_{ij}\neq 0$. We also
postulate that $i\accesses i$. The notion of access is extended to
subsets of $N$, namely, $I\accesses J$ if $i\accesses j$ for some
$i\in I$ and $j\in J$.

Both in max-plus algebra and in nonnegative algebra the Kleene star
series may diverge to infinity (in other words, be unbounded). In
both cases the convergence is strongly related to the largest
eigenvalue of $A$ (w.r.t. the eigenproblem $Ax=\lambda x$), which we
denote by $\rho(A)$. This is also called the {\em Perron root} of
$A$. A necessary and sufficient condition for the convergence is
$\rho(A)<1$ in the case of the ordinary algebra, and $\rho(A)\leq 1$
in the case of the max-times algebra. In the max-times algebra (but
not in the usual algebra) $A^*$ can be always truncated, meaning
$A^*=I+A+\ldots +A^{n-1}$ where $n$ is the dimension of $A$, in the
case of convergence. This is due to the finiteness of $D_A$ and the
optimal path interpretation, see~\cite{BCOQ,PBbook} for more
details.

In the case of max-times algebra, $\rho(A)$ is equal to the
\textit{maximum geometric cycle mean} of $A$, namely
\begin{equation*}
\rho (A) =\max \left\{ \sqrt[k]{%
a_{i_{1}i_{2}}a_{i_{2}i_{3}}...a_{i_{k}i_{1}}};i_{1},...,i_{k}\in
N,k=1,2,...\right\} .
\end{equation*}
This quantity can be computed in $O(n^3)$ time by Karp's algorithm,
see e.g.~\cite{BCOQ,PBbook}.

\subsection{Frobenius normal form}

$A=\left( a_{ij}\right) \in \Rp^{n\times n}$ is called \textit{%
irreducible} if $n=1$ or for any $i,j\in N$ there are $%
i_{1}=i,i_{2},...,i_{k}=j,$ such that $%
a_{i_{1}i_{2}}a_{i_{2}i_{3}}...a_{i_{k-1}i_{k}}>0;$ $A$ is called \textit{%
reducible} otherwise. In other words, a matrix is called irreducible
if the associated graph is strongly connected. Note that if $n>1$
and $A$ is irreducible then $A\neq 0.$ Hence the assumption "$A$
irreducible, $A\neq 0"$ merely means that $A$ is irreducible but not
the $1\times 1$ zero matrix. (It is possible to extend these notions to general semirings
with no zero divisors, but we will not require this in the paper.)

In order to treat the reducible case for max-times algebra and the
(classical) nonnegative linear algebra, we recall some standard
notation and the Frobenius normal form (considering it for general
semirings will be of no use here). If
\begin{equation*}
1\leq i_{1}<i_{2}<...<i_{k}\leq n,K=\{i_{1},...,i_{k}\}\subseteq N
\end{equation*}%
then $A_{KK}$ denotes the \textit{principal submatrix}
\begin{equation*}
\left(
\begin{array}{ccc}
a_{i_{1}i_{1}} & ... & a_{i_{1}i_{k}} \\
... & ... & ... \\
a_{i_{k}i_{1}} & ... & a_{i_{k}i_{k}}%
\end{array}%
\right)
\end{equation*}%
of the matrix $A=(a_{ij})$ and $x_K$ denotes the subvector $%
(x_{i_{1}},...,x_{i_{k}})^{T}$ of the vector $x=(x_{1},...,x_{n})^{T}$.

If $D=(N,E)$ is a digraph and $K\subseteq N$ then $D(K)$ denotes the \textit{%
induced subgraph} of $D,$ that is%
\begin{equation*}
D(K)=(K,E\cap (K\times K)).
\end{equation*}
Observe that $\rho \left( A\right) =0$ if and only if $D_{A}$ is
acyclic.

Every matrix $A=(a_{ij})\in \Rp^{n\times n}$ can be transformed by
simultaneous permutations of the rows and columns in linear time to
a \textit{Frobenius normal form} \cite{Handbook}
\begin{equation}
\left(
\begin{array}{cccc}
A_{11} & 0 & ... & 0 \\
A_{21} & A_{22} & ... & 0 \\
... & ... & ... & ... \\
A_{r1} & A_{r2} & ... & A_{rr}%
\end{array}%
\right) ,  \label{fnf}
\end{equation}%
where $A_{11},...,A_{rr}$ are irreducible square submatrices of $A$,
corresponding to the partition $N_1\cup\ldots\cup N_r=N$ (that is,
$A_{ij}$ is a shortcut for $A_{N_iN_j}$). The sets $N_{1},...,N_{r}$
will be called \textit{classes (of }$A$). It follows that each of
the induced subgraphs $D_{A}(N_{i})$ $(i=1,...,r)$ is strongly
connected and an arc from $N_{i}$ to $N_{j}$ in $D_{A}$ exists only
if $i\geq j.$

If $A$ is in the Frobenius normal form (\ref{fnf}) then the \textit{
reduced graph}, denoted $R(A)$, is the digraph whose nodes are the
classes $N_{1},...,N_{r}$ and the set of arcs is
\begin{equation*}
\{(N_{i},N_{j});(\exists k\in N_{i})(\exists \ell \in N_{j})a_{k\ell }>0\}).
\end{equation*}%
In addition we postulate that each class has a self-loop (useful if
FNF contains trivial classes consisting of one diagonal zero entry).
In the max-times algebra and the nonnegative matrix algebra, the
nodes of $R(A)$ are marked by the corresponding greatest eigenvalues
(Perron roots) $\rho_i:=\rho(A_{ii})$. 


Simultaneous permutations of the rows and columns of $A$ are
equivalent to calculating $P^{-1}AP,$ where $P$ is a generalized
permutation matrix. Such transformations do not change the
eigenvalues, and the eigenvectors before and after such a
transformation only differ by the order of their components. So when
solving the eigenproblem, we may assume without loss of generality
that $A$ is in a Frobenius normal form, say (\ref{fnf}).

\subsection{Eigenvalues and eigenvectors in max-times algebra}
\label{ss:sptheory}

 It is intuitively clear that all eigenvalues of $A$ are
among the unique eigenvalues of diagonal blocks. However, not all of
these eigenvalues are also eigenvalues of $A.$ The following key
result, describing the set $\Lambda(A)$ of all eigenvalues of $A$ in
{\bf max-times algebra} (that is, set of all $\lambda$ such that
$Ax=\lambda x$ has a nonzero solution $x$ in max-times algebra)
appeared for the first time independently in Gaubert's thesis
\cite{gaubertthesis} and Bapat~et~al.~\cite{bapatfirst}, see also Butkovi\v{c}~\cite{PBbook}.

\begin{theorem}[(cf. \cite{bapatfirst,PBbook,gaubertthesis})]
\label{spectral} Let (\ref{fnf}) be a Frobenius
normal
form of a matrix $A\in \mathbb{R}_{+}^{n\times n}.$ Then%
\begin{equation}
\label{e:spcond} \Lambda (A)=\{\rho_j;\rho_j\geq\rho_i\ \text{for
all $N_i\accesses N_j$}\}.
\end{equation}
\end{theorem}

The same result holds in the nonnegative linear algebra, with the
non-strict inequality replaced by the strict one.

If a diagonal block $A_{jj}$ has $\rho_j\in \Lambda $, it still may
not satisfy the condition in Theorem~\ref{spectral} and may
therefore not provide any eigenvectors. So it is necessary to
identify classes $j$ that satisfy this condition and call them
\textit{spectral}. Thus $\rho_j\in \Lambda (A)$ if $N_j$ is
spectral, but not necessarily the other way round. We can
immediately deduce that all initial blocks are spectral, like in the
nonnegative linear algebra. Also, it follows that the number of
eigenvalues does not exceed $n$ and obviously, $\rho(A)
=\max_i\rho_i$, in accordance with $\rho(A)$ being the greatest
eigenvalue.

We are now going to describe, for $\lambda\in\Lambda$, the eigencone
$V(A,\lambda)$ of all vectors $x$ such that $Ax=\lambda x$. Denote
by $J_{\lambda}$ the union of all classes $N_i$ which have access to
the spectral classes corresponding to this eigenvalue.
By~(\ref{e:spcond}), $\rho_i\leq\lambda$ for all such classes. Now
we define the {\em critical graph} $C_A(\lambda)=(N_c,E_c)$
comprising all nodes and edges on {\em critical cycles} of the
submatrix $A_{J_{\lambda}J_{\lambda}}$, i.e., such cycles where
$\lambda$ is attained. This graph consists of several strongly
connected components, and let $T(A,\lambda)$ denote a set of indices
containing precisely one index from each component of
$C_A(\lambda)$. In the following, $A'(J_{\lambda})$ will denote the
$n\times n$ matrix, which has $A_{J_{\lambda}J_{\lambda}}/\lambda$
as submatrix, and zero entries everywhere else.

\begin{theorem}[(cf. \cite{bapatfirst,PBbook,gaubertthesis})]
\label{t:red-spec} Let $A\in\Rp^{n\times n}$ and
$\lambda\in\Lambda(A)$. Then
\begin{enumerate}
\item[a)] For any eigenvector $v\in V(A,\lambda)$ there exist
$\alpha_j\in\Rp$ such that $v$ is the max-times linear combination
\begin{equation}
\label{generate}
 v=\sum_{j\in T(A,\lambda)} \alpha_j
(A'(J_{\lambda}))^*_{\cdot j}.
\end{equation}
\item[b)] For any two indices $j$ and $k$ in the same component of
$C_A(\lambda)$, columns $(A'(J_{\lambda}))^*_{\cdot j}$ and
$(A'(J_{\lambda}))^*_{\cdot k}$ are proportional.
\item[c)] Vectors $(A'(J_{\lambda}))^*_{\cdot
j}$ for $j\in T(A,\lambda)$ form a basis of $V(A,\lambda)$, that is,
they generate $V(A,\lambda)$ in the sense of a) and none of them can
be expressed as a max-times linear combination of the others.
\end{enumerate}
\end{theorem}

\begin{remark}
{\rm An analogous description of $V(A,\lambda)$ in
nonnegative matrix algebra is called Frobenius-Victory theorem
\cite{Fro-1912,Vic-85}, see \cite{Sch-86} Theorem 3.7. Namely, to
each spectral node of $R(A)$ with eigenvalue $\lambda$, there
corresponds a unique eigenvector, with support equal to the union of
classes having access to the spectral node. These eigenvectors are
the extreme rays of the cone, i.e. they form a "basis" in analogy
with Theorem~\ref{t:red-spec}.

Moreover, these extreme rays are linearly independent as it may be
deduced from their supports. In parallel, it can be shown that the
generators of Theorem~\ref{t:red-spec} are strongly regular,
see~\cite{But-03} for definition (i.e., independent in a stronger
sense).

However, extremals in the nonnegative case do not come from
$A^*=(I-A)^{-1}$ and, to the authors' knowledge, no explicit
algebraic expression for these vectors is known.}
\end{remark}

\section{Theory of Z-matrix equations}

\subsection{General results}
\label{ss:general}

In the following, we describe the solution set of $x=Ax+b$ as
combinations of the least solution $A^{\ast}b$ and the eigenvectors
of $A$. The results of this subsection hold for the max-times
algebra, nonnegative linear algebra, and an abstract algebraic setup
(reassuring that $A^*b$ satisfies $x=Ax+b$, and $v=\inf\limits_k A^k
x$, for $x$ such that $Ax\leq x$, satisfies $Av=v$), which will be
provided in Section~\ref{s:algen}.

We start with a well-known fact, that
\begin{equation}
\label{bellman-def} A^{\ast}b:=\sup_k(b+Ab+A^2b+\ldots+A^{k-1}b)
\end{equation}
is the least solution to $x=Ax+b$. We will formulate it in the form
of an equivalence. Note that the supremum~(\ref{bellman-def}) may
exist even if $A^*$ does not exist. (In this sense, $A^*b$ is rather
a symbol than a result of matrix-vector multiplication. On the other
hand, one can complete a semiring with the greatest element
"$+\infty$" and regard $A^*b$ as a matrix-vector product.)

\begin{theorem}[(Well-known, cf.~\cite{BCOQ,GM:08})]
\label{t:equivalence}Let $A\in \semr^{n\times n},b\in \semr^{n}$.
The following are equivalent:
\begin{enumerate}
\item[(i)] $x=Ax+b$ has a solution,
\item[(ii)] $x=Ax+b$ has a least solution.
\item[(iii)] $A^{\ast}b$ converges.
\end{enumerate}
If any of the equivalent statements holds, $A^{\ast}b$ is the least
solution of $x=Ax+b$.
\end{theorem}
\begin{proof}
(i)$\Rightarrow$ (iii) Let $x$ be a solution to $x=Ax+b$. Then
\begin{eqnarray*}
x &=&Ax+b \\
&=&A\left( Ax+b\right) +b \\
&=&A\left( A\left( Ax+b\right) +b\right) +b=....
\end{eqnarray*}%
Therefore for any $k\geq 1$ we have%
\begin{equation}
x=A^{k}x+\left( A^{k-1}+A^{k-2}+...+I\right) b. \label{x expressed}
\end{equation}
This shows that the expressions in~\eqref{bellman-def} are bounded
from above by $x$, hence the supremum exists.

(iii)$\Rightarrow$ (ii) We verify that
\begin{equation}
\label{Aastb}
\begin{split}
&AA^{\ast}b+b=A\sup_k(b+Ab+\ldots+A^{k-1}b)+b\\
&=\sup_k(b+Ab+\ldots+A^kb)=A^{\ast}b,
\end{split}
\end{equation}
treating sup as a limit and using the continuity of the
matrix-vector multiplication.

(From the algebraic point of view, we used the distributivity of
(max-algebraic, nonnegative) matrix multiplication with respect to
sup's of ascending chains, and the distributivity of $+$ with respect to such 
sup's. Further details on this will be given in Section~\ref{s:algen}.)

(ii)$\Rightarrow$ (i) Trivial.
\end{proof}

We proceed with characterizing the whole set of solutions. (See also
Remark~\ref{r:krivulin} for an alternative short proof of the first part.) 
\begin{theorem}
\label{Th main 1}Let $A\in \semr^{n\times n},b\in \semr^{n}$ be such
that~$x=Ax+b$ has a solution. Then
\begin{enumerate}
\item[(a)] The set of all solutions to $x=Ax+b$ is $\left\{
v+A^{\ast}b;Av= v\right\} $;
\item[(b)] If for any $x$ such that $Ax\leq x$
we have $\inf\limits_k A^k x =0$, then
 $A^{\ast }b$ is the unique solution to $x=Ax+b$.
\end{enumerate}
\end{theorem}

%

\begin{proof} (a) Firstly we need to verify that any vector of the form
$v+A^{\ast}b$ where $v$ satisfies $Av=v$, solves~(\ref{Eq Ax+b=x}).
Indeed,
\begin{equation*}
A(v+A^{\ast}b)+b=Av+(AA^{\ast}b+b)=v+A^{\ast}b,
\end{equation*}
where we used that $Av=v$ and that $A^{\ast}b$ is a solution
of~(\ref{Eq Ax+b=x}), see Theorem~\ref{t:equivalence}. It remains to
deduce that each solution of~(\ref{Eq Ax+b=x}) is as defined above.

Let $x$ be a solution to (\ref{Eq Ax+b=x}), and denote
$y^{(k)}:=A^{k}x$ and $$z^{(k)}:=\left( A^{k-1}+A^{k-2}+...+I\right)
b.$$ We have seen in~(\ref{x expressed}) that 
\begin{equation}
\label{forallk}
x=y^{(k)}+z^{(k)},\quad \text{for all $k\geq 1$}.
\end{equation}
Since $Ax\leq x$ it follows that the sequence $y^{(k)}$ is nonincreasing.
The sequence of $z^{(k)}$ is nondecreasing. 

{\em Both in max-times and in the nonnegative case}, we conclude that
$v=\lim_{k\to\infty} y^{(k)}$ exists and (by the continuity of $A$ as operator) we have $Av=v$.
We also obtain that $A^*b=\lim_{k\to\infty} z^{(k)}$, and finally
\begin{equation}
x=\lim_{k\to\infty} y^{(k)}+\lim_{k\to\infty} z^{(k)}=v+ A^*b,
\end{equation}  
where $v$ satisfies $Av=v$. {\em The theorem is proved,} both for max-times algebra
and nonnegative linear algebra.

{\em In a more general semiring context (see Section~\ref{s:algen})},
it remains to show that
$\Tilde{x}:=\inf_k y^{(k)} + \sup_k z^{(k)}$ is the same as
$y^{(k)}+z^{(k)}$ for all $k$. After showing this we are done, since
$\sup\limits_k z^{(k)}=A^*b$, and also
\begin{equation}
\label{infdist-mat}
 A\inf_{k\geq 0} y^{(k)}=A(\inf_{k\geq 0} A^k x)= \inf_{k\geq
1} A^k x =\inf\limits_{k\geq 0} y^{(k)},
\end{equation}
so that we can set $v:=\inf_k y^{(k)}$, it satisfies $Av=v$. (From the algebraic point of view, we have used the distributivity
of matrix multiplication with respect to inf's of descending chains. Further details will
be given in Section~\ref{s:algen}.)

Using the distributivity of $+$ with
respect to $\inf$ we obtain
\begin{equation}
\label{xtildegeqx}
\Tilde{x}=\inf_k (y^{(k)}+\sup_l z^{(l)})\geq x,
\end{equation}
since this is true of any term in the brackets.  Using the
distributivity with respect to $\sup$ we obtain
\begin{equation}
\label{xtildeleqx}
\Tilde{x}=\sup_l (\inf_k y^{(k)}+ z^{(l)})\leq x,
\end{equation}
for analogous reason. Combining~\eqref{xtildegeqx} and~\eqref{xtildeleqx} we obtain
$$x=\Tilde{x}=\inf_k y^{(k)} + \sup_k z^{(k)}=v+A^{\ast}b,$$
which yields a {\bf general proof} of part (a). 

For part (b), recall that $y^{(k)}:=A^k x$,
and that $x$ satisfies $Ax\leq x$.
\end{proof}

These results also apply to equations $\lambda x= Ax+b$ when
$\lambda$ is invertible: it suffices to divide this equation by
$\lambda$.
\if{
\begin{remark}
{\rm Both in max-times algebra and in the classical nonnegative
algebra, condition $\rho(A)<1$ is sufficient for $\inf_k A^k x=0$
(for any $x$), hence it is also sufficient for the uniqueness of the
solution of $x=Ax+b$.}
\end{remark}
}\fi

\begin{remark}
\label{r:convexity} {\rm The solution set of $x=Ax+b$ is convex over
$\semr$, since it contains with any two points $x,y$ all convex
combinations $\lambda x+\mu y$, $\lambda+\mu=1$. Further, {\em both in
max-times semiring and in the nonnegative algebra},
$A^{\ast}b$ is the only extreme point: it cannot be a convex
combination of two solutions different from it. The eigenvectors of
$A$ are {\em recessive rays},i.e., any multiple of such vectors can
be added to any solution, and the result will be a solution again.
Moreover, only eigenvectors have this property. Indeed, assume that
$z$ is a solution, $z+\mu v$ where $\mu\neq 0$ satisfies
\begin{equation*}
z+\mu v = A(z+\mu v) +b,
\end{equation*}
but $Av\neq v$. In the usual algebra this is impossible. In
max-times, assume that $(Av)_i\neq v_i$ for some $i$, then one of
these is nonzero. As $z_i=(Az)_i+b_i$ is finite, taking large enough
$\mu$ will make $\mu v_i$ or $\mu (Av)_i$ the only maximum on both
l.h.s. and r.h.s., in which case $z+\mu v$ will not be a solution.
Thus, in both theories the eigencone of $A$ with eigenvalue $1$ is
the recessive cone of the solution set of $x=Ax+b$. In the max-times
case it is generated by the fundamental eigenvectors as in
Theorem~\ref{t:red-spec}. Thus we have an example of the tropical
theorem of Minkowski, representing closed max-times convex sets in
terms of extremal points and recessive rays, as proved by Gaubert
and Katz~\cite{GK-07}}.
\end{remark}

\begin{remark}
\label{r:krivulin} {\rm In this remark we recall the proof of Theorem~\ref{Th main 1} part a) given by
Krivulin, see~\cite{Kri-06} Lemma 7 or~\cite{Kri:09} Lemma 3.5. We slightly change
the original proof to make it work also for nonnegative linear algebra. Let $x$ be a solution of $x=Ax+b$ and
define $\underline{w}$ as the least vector $w$ satisfying $x=u+w$ where 
$u:=A^*b$. It can be defined explicitly by
\begin{equation}
\label{leastwdef}
\underline{w}_i=
\begin{cases}
x_i,& \text{if $x_i>u_i$},\\
0,& \text{if $x_i=u_i$},
\end{cases}
\; \text{or}\quad
\underline{w}_i=
\begin{cases}
x_i-u_i,& \text{if $x_i>u_i$},\\
0,& \text{if $x_i=u_i$},
\end{cases},
\end{equation}
in the case of max algebra and nonnegative linear algebra respectively.
Now notice that if $x=u+w$ then $x=u+Aw$. Indeed
\begin{equation*}
x=A(u+w)+b=(Au+b)+Aw=u+Aw.
\end{equation*}
Hence $\underline{w}\leq A\underline{w}$. Indeed, both $w:=\underline{w}$ and $w:=A\underline{w}$ satisfy
$x=u+w$ but $\underline{w}$ is the least such vector. Defining $\underline{v}:=\sup_{n\geq 1} A^n\underline{w}$
we obtain $x=u+\underline{v}$ and $A\underline{v}=\underline{v}$. 
The algebraic generality of this argument is also quite high, it will be discussed in the last
section of the paper.}
\end{remark}

\subsection{Spectral condition, Frobenius trace-down method}

We consider equation~$\lambda x=Ax+b$ in {\bf max-times algebra and nonnegative linear algebra},
starting with the case when $A$ is irreducible.  Theorem~\ref{Th
main 2} below can be also viewed in nonnegative linear algebra, but
only after some modification which will be described. Denote
$A_{\lambda}:=A/\lambda$ and $b_{\lambda}:=b/\lambda$.

The following is a {\bf max-times} version of the Collatz-Wielandt
identity in the Perron-Frobenius theory.

\begin{lemma}[(Well-known, cf.~\cite{PBbook,gaubertthesis})]
\label{wellknown}Let $A\in \mathbb{R}_{+}^{n\times n},A\neq 0.$ Then $%
Ax\leq \lambda x$ has a solution $x>0$ if and only if $\lambda \geq
\rho \left( A\right) ,\lambda >0.$
\end{lemma}

\begin{proof}
Let $x>0$ be a solution, then $Ax\neq 0$ and so $\lambda >0.$ If
$\rho \left( A\right) =0$ there is nothing to prove, so we may
suppose $\rho \left( A\right) >0.$ Let $\sigma =\left(
i_{1},...,i_{k},i_{k+1}=i_{1}\right) $ be any cycle with nonzero
weight. Then
\begin{eqnarray*}
a_{i_{1}i_{2}}x_{i_{2}} &\leq &\lambda x_{i_{1}} \\
a_{i_{2}i_{3}}x_{i_{3}} &\leq &\lambda x_{i_{2}} \\
&&... \\
a_{i_{k}i_{1}}x_{i_{1}} &\leq &\lambda x_{i_{k}}.
\end{eqnarray*}%
After multiplying out and simplification we get $\lambda \geq \sqrt[k]{%
a_{i_{1}i_{2}}a_{i_{2}i_{3}}...a_{i_{k}i_{1}}}$ and so $\lambda \geq
\rho \left( A\right) .$

Suppose now $\lambda \geq \rho \left( A\right) ,\lambda >0.$ Then
$\rho \left( A_{\lambda }\right) \leq 1$ and so $A_{\lambda }^{\ast
}=I+A_{\lambda }+...+A_{\lambda }^{k}$ for every $k\geq n-1,$
yielding $A_{\lambda }A_{\lambda }^{\ast }\leq A_{\lambda }^{\ast
}.$ Let $u$ be any positive vector in $\mathbb{R}_{+}^{n}.$ Take
$x=A_{\lambda }^{\ast }u,$ then $x>0$ because $A_{\lambda }^{\ast
}u\geq u$ and
\begin{equation*}
A_{\lambda }x=A_{\lambda }A_{\lambda }^{\ast }u\leq A_{\lambda
}^{\ast }u=x.
\end{equation*}
\end{proof}

\begin{lemma}[(Well-known, cf.~\cite{BCOQ, Ostrowski})]
\label{oracle} If $A\in \mathbb{R}_{+}^{n\times n}$
is irreducible, $b\in \mathbb{R}_{+}^{n},b\neq 0$ and $\lambda >0$ then the
following are equivalent:
\begin{enumerate}
\item[(i)] $\lambda x=Ax+b$ has a solution.
\item[(ii)] $A_{\lambda}^*$ converges.
\item[(iii)] $\lambda\geq\rho(A)$ (max-times algebra), 
$\lambda>\rho(A)$ (nonnegative linear algebra). 
\end{enumerate}
All solutions of $\lambda x= Ax+b$ (if any) are positive.
\end{lemma}
\begin{proof}
In the case of nonnegative matrix algebra, this Lemma follows from the results of
Ostrowski's famous paper~\cite{Ostrowski} (where matrices appear as determinants),
see also~\cite{Sch-56}, Lemma 5. For the equivalence between (ii) and (iii) in max-times
algebra, consult e.g. \cite{BCOQ,PBbook,GM:08}. (Both in max-times algebra and in the nonnegative linear algebra,
such equivalence holds also for reducible matrices.) For the reader's convenience we show the
equivalence between (i) and (iii) in max-times algebra.

(iii)$\Rightarrow$(i): If $\lambda \geq \rho \left( A\right) $ then
$1\geq\rho(A_{\lambda})$, hence $A_{\lambda}^*$ and
$A_{\lambda}^*b_{\lambda}$ converge. In this case,
$A_{\lambda}^*b_{\lambda}$ is the least solution by
Theorem~\ref{t:equivalence}.

(i)$\Rightarrow$(iii): If $Ax+b=\lambda x$ then $\lambda >0$ and
$x\neq 0$ since $b\neq 0.$ We need to show that $x>0$, to apply
Lemma~\ref{wellknown}.

If $n=1$ then the result holds. Suppose now $n>1$, thus $\rho \left(
A\right)>0.$ Let $B=\left( \rho \left( A\right) \right) ^{-1}A$ and
$\mu =\left( \rho \left( A\right) \right) ^{-1}\lambda .$ Then $B$
has $\rho(B)=1$, it is irreducible and $Bx\leq \mu x.$ Therefore
$B^{*}>0,$ thus $B^{*}x>0$. But $B^*x\leq \mu x$ and hence $x>0$. By
Lemma~\ref{wellknown} we obtain that $\lambda\geq\rho(A)$.
\end{proof}

\begin{remark}
\label{Rem on infinity} {\rm Note that also for general (reducible)
$A$, if $b>0$ then for any solution $x$ of $\lambda x=Ax+b$ we have
$x>0,$ and hence $\lambda \geq \rho \left( A\right) $ by
Lemma~\ref{oracle}. However, this condition is not necessary for the
existence of a solution to $Ax+b=\lambda x,$ when $b$ has at least one zero component,
see Theorem \ref{Th main 2} below. If $%
\lambda <\rho(A)$ then some entries of $A_{\lambda }^{\ast }$ are
$+\infty $ and it is not obvious from Theorem \ref{Th main 1}
whether a finite solution exists since the product $A_{\lambda
}^{\ast }b_{\lambda }$
may in general (if $\lambda $ is too low) contain $+\infty $. However if $%
0.\left( +\infty \right) $ is defined as $0$ and the condition of
Theorem \ref{Th main 2} (iv) holds, then
the $+\infty $ entries of $A_{\lambda }^{\ast }$ in $%
A_{\lambda }^{\ast }b_{\lambda }$ will always be matched with zero components of $%
b_{\lambda },$ and consequently $A_{\lambda }^{\ast }b_{\lambda }$
will be a finite non-negative vector.}
\end{remark}

Now we consider the general (reducible) case. The next result
appears as the main result of this paper, describing the solution
sets to $Z$-equations in max-times algebra and nonnegative linear algebra. 

\begin{theorem}
\label{Th main 2} Let $A\in \mathbb{R}_{+}^{n\times n}$ be in FNF with classes $N_j, j=1,\ldots,s$. Let $b\in \mathbb{R%
}_{+}^{n},\lambda \geq 0,$. Denote $J=\left\{ j;N_j\accesses \operatorname{supp}%
\left( b\right) \right\}$ and $\overline{\rho }= \max_{j\in
J}\rho_{j}$ (for the case when $b=0$ and $J=\emptyset$ assume that
$\max\emptyset = 0$). The following are equivalent:
\begin{enumerate}
\item[(i)] System $\lambda x=Ax+b$ has a solution.
\item[(ii)] System $\lambda x=Ax+b$ has the least solution.
\item[(iii)] $x^{(0)}=A_{\lambda}^{\ast}b_{\lambda}$ converges.
\item[(iv)] If $j\in J$ then $(A_{jj})^*_{\lambda}$ converges.
\item[(v)] $\overline{\rho}\leq \lambda$ (max-times), or $\overline{\rho}<\lambda$ (nonnegative linear algebra)
\end{enumerate}
If any of the equivalent statements hold, then
\begin{enumerate}
\item[a)] $x^0$ is the least solution of $\lambda x=Ax+b$. For this
solution, $x^0_{N_i}\neq 0$ when $i\in J$ and $x^0_{N_i}=0$ when
$i\notin J$. The solution $x^0$ is unique if and only if $\lambda$
is not an eigenvalue of $A$.
\item[b)] Any solution $x$ of (\ref{Eq Ax+b=lx}) can be expressed as
$x=x^0+v$ where $v$ satisfies $Av=\lambda v$.
\end{enumerate}
\end{theorem}
\begin{proof} We first treat the trivial case, when $b=0$. In this case $x^0=0$ is
a solution, $J=\emptyset ,\overline{\rho }=0\leq\lambda$ and thus
all the equivalent statements (i)-(iv) are true; (a) and (b) hold
trivially with $x^0=0.$

We now suppose $b\neq 0.$ Consequently, $\lambda >0,$ and assume
w.l.o.g. that $\lambda=1$. The equivalence of (i)-(iii) was
manifested in Theorem~\ref{t:equivalence}, and part b) was proved in
Theorem~\ref{Th main 1}. The equivalence of (iv) and (v) follows from Lemma~\ref{oracle}.
It remains to show the equivalence of (i)
and (iv), that the minimal solution has a prescribed support, and
the spectral criterion for uniqueness.

We show that (i) implies (iv). For simplicity we use the same symbol
"$J$" for the set of indices in the classes of $J$. Denote
$I:=\{1,\ldots,n\}\backslash J$. We have
\begin{equation}
\label{e:IJ}
\begin{pmatrix}
x_I\\
x_J
\end{pmatrix}=
\begin{pmatrix}
A_{II} & 0\\
A_{JI} & A_{JJ}
\end{pmatrix}
\begin{pmatrix}
x_I\\
x_J
\end{pmatrix}
+
\begin{pmatrix}
0\\
b_J
\end{pmatrix},
\end{equation}
and hence $x_I$ is a solution of $A_{II}x_I=x_I$, and $x_J$ is a
solution of  $x_J=A_{JJ}x_J+A_{JI}x_I+b_J$. Further, denote
$\Tilde{b}_J:=A_{JI}x_I+b_J$ and let $J$ consist of the classes
$N_1,\ldots,N_t$ ($t\leq s$). Then
\begin{equation*}
A(J)=\left(
\begin{tabular}{cccc}
$A_{11}$ & 0 & 0 & 0 \\
$A_{21}$ & $A_{22}$ & 0 & 0 \\
$\cdots $ & $\cdots $ & $\ddots $ & 0 \\
$A_{t1}$ & $A_{t2}$ & $\cdots $ & $A_{tt}$%
\end{tabular}%
\right) .
\end{equation*}%

We now proceed by an inductive argument, showing that $(A_{jj})^*$ converges
for all $j=1,\ldots,t$, and that all components in
$x_{N_1},\ldots,x_{N_t}$ are positive. This argument is a
max-algebraic version of the {\bf Frobenius trace-down method}.

As the {\bf base of induction}, we have $x_{N_1}= A_{11}
x_{N_1}+ \Tilde{b}_{N_1}$. In this case, the class $N_1$ is final,
so $b_{N_1}$ and hence $\Tilde{b}_{N_1}$ should have some positive
components. Using Lemma~\ref{oracle}, we conclude that $(A_{11})^*$ 
converges and $x_{N_1}$ is positive.

{\bf Induction step.} Suppose that for some $l$, all components of
$(x_{N_1},\ldots,x_{N_l})$ solving
\begin{equation}
\label{lchain}
\begin{split}
x_{N_1} & =  A_{11} x_{N_1} + \Tilde{b}_{N_1}\\
x_{N_2} & =  A_{21} x_{N_1} + A_{22} x_{N_2} + \Tilde{b}_{N_2}\\
\ldots & = \ldots\\
x_{N_l} & =  A_{l1} x_{N_1} +\ldots + A_{ll} x_{N_l} +
\Tilde{b}_{N_l},
\end{split}
\end{equation}
are positive. We show that the same holds if we add the next
equation
\begin{equation}
\label{lplus1} x_{N_{l+1}} = A_{l+1,1}x_{N_1}+\ldots+A_{l+1,l+1}
x_{N_{l+1}}+\Tilde{b}_{N_{l+1}},
\end{equation}
and that $(A_{l+1,l+1})^*$ converges. We have two cases: either $b_{N_{l+1}}$
has nonzero components so that $N_{l+1}$ intersects with $\supp(b)$,
or if not, $N_{l+1}$ should access a class which intersects with
$\supp(b)$. In this case, one of the submatrices $A_{l+1,1},\ldots,
A_{l+1,l}$ is not identically $0$. As all components of
$x_{N_1},\ldots, x_{N_l}$ are positive, this shows that the sum on
the r.h.s. of~\eqref{lplus1} excluding $A_{l+1,l+1}x_{N_{l+1}}$ has some
positive components in any case. Using Lemma~\ref{oracle}, we
conclude that $(A_{l+1,l+1})^*$ converges and $x_{N_{l+1}}$ is positive.

Now we show that (iv) implies (i), and moreover, that there is a
solution with prescribed support structure. To do so, we let $x_I=0$
in~\eqref{e:IJ}. Then it is reduced to~$x_J=A_{JJ}x_J+b_J$, and we
have to show the existence of a positive solution $x_J$. The proof
of this follows the lines of the Frobenius trace-down method
described above, making the inductive assumption that~(\ref{lchain})
has a positive solution $(x_{N_{1}},\ldots,x_{N_l})$ and using
Lemma~\ref{oracle} to show that~(\ref{lplus1}) can be solved with a
positive $x_{N_{l+1}}$. Strictly speaking, in this case we have $b$
instead of $\Tilde{b}$ in~(\ref{lchain}) and~(\ref{lplus1}), but
this does not make any change in the argument.

Let the conditions (i)-(v) be satisfied.
Since letting $x_I=0$ in~\eqref{e:IJ} produces a solution (see above), 
the support of the least solution is contained in
$J$. However, the support of {\em any} solution should contain $J$
by the argument in the proof of (i)$\Rightarrow$(iv).

Evidently, solution $x^0$ is unique if $\lambda$ is not an
eigenvalue of $A$. To show the converse (in max-times algebra), note
that for any nonzero $v$ there is a large enough $\alpha$ such that
some component $\alpha v_i$ is greater than $x^0_i$, hence
$x^0+\alpha v\neq x^0$. (Note that the converse would be evident in
the usual nonnegative algebra.)

The proof is complete.
\end{proof}

\begin{remark}[(cf \cite{BC-75,LMa-00})]
{\rm The Frobenius trace-down method of Theorem~\ref{Th main 2} can
be also viewed as a generalized block-triangular elimination
algorithm for obtaining the least solution $A^*b$ (assumed w.l.o.g.
that $\lambda=1$). Namely, if $(x_{N_1}\ldots x_{N_l})$ is the least
solution of~(\ref{lchain}), then computing
\begin{equation}
\label{lplus2} x_{N_{l+1}}:=
(A_{l+1,l+1})^{\ast}(A_{l+1,1}x_{N_1}+\ldots+A_{l+1,l}
x_{N_l}+b_{N_{l+1}})
\end{equation}
yields the least solution $(x_{N_1},\ldots,x_{N_l},x_{N_{l+1}})$ of
the enlarged system~(\ref{lchain}) \& (\ref{lplus1}) with $b$
instead of $\Tilde{b}$. Indeed, if we suppose that there is another
solution $(x'_{N_1},\ldots,x'_{N_{l+1}})$, then $x'_{N_i}\geq
x_{N_i}$, and it follows from~(\ref{lplus2}) that $x'_{N_{l+1}}\geq
x_{N_{l+1}}$. As an algorithm for finding the least solution of
$x=Ax+b$, it is valid even for more general semirings than the
setting of Section~\ref{s:algen}, provided that a solution to
$x=Ax+b$ exists.}
\end{remark}

\begin{remark}
{\rm The description of the support of $x^0$ as in Theorem~\ref{Th
main 2} a) can be obtained directly from the path interpretation of
$A^*$, using that $x^0=A^{\ast} b$ (when this is finite). Indeed,
write $b=\sum_{k\in\supp(b)} \beta_k e_k$ where $e_k$ is the $k$th
unit vector and $\beta_k$ are all positive. Hence $x^0=A^{\ast}b
=\sum_{k\in\supp(b)} \beta_k A^{\ast}_{\cdot k}$.  It can be now
deduced from the path interpretation of $A^{\ast}$, that $x^0_l>0$
whenever $l$ accesses $k$ from $\supp(b)$. This argument also shows
that the description of the support of $x^0=A^{\ast}b$ is valid over
any semiring with no zero divisors. With zero divisors, the access
condition for $x^0_l>0$ may be no longer sufficient.}
\end{remark}

We have represented any solution $x$ of $x=Ax+b$ in the form
$x^0+v$, where $x^0=A^*b$ and $v=\inf_k A^k x$. Below we give an
explicit formula for $v$ in the case of max-times algebra, due to
Dhingra and Gaubert~\cite{DG-06}. \if{ An explicit expression for
$v$ in the case of max-min algebra will be described in
Section~\ref{s:algen}.}\fi

Let $C$ be the set of critical nodes (i.e., nodes of the critical
graph) corresponding to the eigenvalue $1$ (see
Subsection~\ref{ss:sptheory}). For any critical cycle
$(i_1,\ldots,i_k)$ we either obtain $x_{i_l}=0$ for all
$l=1,\ldots,k$, or both $x_{i_l}\neq 0$ and
$a_{i_li_{l+1}}x_{i_{l+1}}=x_{i_l}$ for all $l$ (after multiplying
all inequalities $a_{i_li_{l+1}}x_{i_{l+1}}\leq x_{i_l}$ and
canceling $x_{i_1}\ldots x_{i_k}$ it turns out that any strict
inequality causes $a_{i_1i_2}\ldots a_{i_ki_1}<1$). This implies
$(Ax)_C=x_C$, for the critical subvectors of $Ax$ and $x$. Applying
$A$ to $x^{(k)}=A^{(k)}x$ which also satisfies $Ax^{(k)}\leq
x^{(k)}$ we obtain that $(A^kx)_C=x_C$ for any $k$, and hence also
$v_C=x_C$.

It remains to determine the non-critical part of $v$. For this we
expand the non-critical part of $Av=v$ as $v_N = A_{NC} v_C + A_{NN}
v_N$. Forming $A_{NN}$ corresponds to eliminating all spectral nodes
of eigenvalue $1$ from the reduced graph $R(A)$. The non-spectral
nodes with eigenvalue $1$ will remain non-spectral, hence $A_{NN}$
does not have eigenvalue $1$, and $(A_{NN})^*A_{NC}v_C$ is the only
solution. Combining with the previous argument we conclude that
\begin{equation}
\label{v expressed} v_C=x_C,\quad v_N=(A_{NN})^*A_{NC}x_C.
\end{equation}

\subsection{Max-times algebra and nonnegative linear algebra}
\label{ss:maxnonneg}

We can make further comparison with the survey of Schneider~\cite{Sch-86} Sect.~4, and
with the work of Hershkowitz and
Schneider~\cite{HSHershkowitz} describing solutions of Z-matrix
equations in nonnegative linear algebra. It can be seen that:

\begin{enumerate}

\item[1.] In the nonnegative linear algebra, Theorem~\ref{Th main 2} extends the statements of \cite{Sch-86} Theorem 4.3
and \cite{Sch-86} Theorem 4.12 (due to Victory~\cite{Vic-85}). In particular, it gives an explicit
formula for the least solution.  

\item[2.] Frobenius trace-down method is also used in the proof
of~\cite{HSHershkowitz} Proposition 3.6, and condition (v) of
Theorem~\ref{Th main 2} (with the strict inequality) is equivalent
to~\cite{HSHershkowitz} condition (3.12).

\item[3.] As observed in~\cite{HSHershkowitz} Theorem 3.16, in the
case of nonnegative algebra, $x^0$ is the only vector with support
$J$, because $\supp(b)$ cannot be accessed by the spectral classes
of $R(A)$ in the case of solvability. However, this is not the case
in the max-times algebra, making it possible that all spectral classes of $R(A)$ 
access $\supp(b)$. This condition is necessary and sufficient for all solutions of $\lambda x=Ax+b$ to 
have the same support as $A^*b$.

\if{
\item[3.] In the case of nonnegative algebra, the spectral nodes with
$\rho_i=\lambda$ are easily identified by the condition of
Theorem~\ref{spectral} (with strict inequality). However, unlike in
the case of max-times algebra, the corresponding Perron eigenvectors
may not generate the whole cone of nonnegative eigenvectors, because
the blocks with higher Perron roots may also have eigenvalue
$\lambda$. The corresponding eigenvectors have some negative
components, but they still contribute to the cone. The extremal rays
of the nonnegative eigenvector cone are the vectors of preferred
basis described in~\cite{HSHershkowitz} Theorem 3.20. }\fi

\item[4.] It can be observed that geometric and combinatorial
descriptions of the solution set in the usual nonnegative algebra,
as provided by~ \cite{HSHershkowitz} Theorem 3.20 and Corollary
3.21, can be deduced from~Theorem~\ref{Th
main 2}, with an application of
Frobenius-Victory theorem, see remark after Theorem~\ref{t:red-spec}. Max-times
analogues of these results of~\cite{HSHershkowitz} can be also
easily formulated.
\end{enumerate}

We next give an example illustrating similarity and difference
between the two theories. Let $A$ be the matrix
\[\begin{pmatrix}
1 &  0 &  0 &  0 &  0 &  0 &  0 \\
0 &  1 &  0 &  0 &  0 &  0 &  0 \\
1 &  0 &  1 &  0 &  0 &  0 &  0 \\
0 &  1 &  0 &  0 &  0 &  0 &  0 \\
0 &  1 &  0 &  0 &  0 &  0 &  0 \\
0 &  0 &  1 &  1 &  0 &  0 &  0 \\
0 &  0 &  0 &  0 &  1 &  0 &  2
\end{pmatrix} \]
We note that this matrix is essentially the same as
in~\cite{HSHershkowitz} Example 3.22, that is, we have replaced
$I-A$ by $A$ and its (reduced) graph $R(A)$, given below, differs
from the one in that Example 3.22 only by the markings of the nodes.

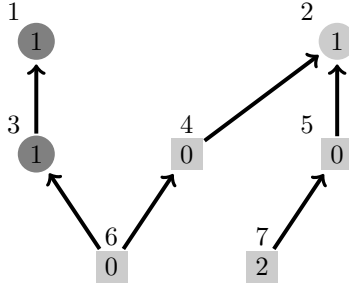
\begin{figure}[h]
\begin{center}
\begin{tikzpicture}[auto, line width=1.5pt, shorten >=1pt,->]

\tikzstyle{subspectral1}=[rectangle,fill=black!20,minimum
size=12pt,inner sep=1pt]
\tikzstyle{subspectral2}=[rectangle,fill=black!20,minimum
size=12pt,inner sep=1pt]

\tikzstyle{superspectral}=[rectangle,fill=black!20,minimum
size=12pt,inner sep=1pt]

\tikzstyle{spectral}=[circle,fill=black!50,minimum size=14pt,inner
sep=1pt]

\tikzstyle{pseudospectral}=[circle,fill=black!20,minimum
size=14pt,inner sep=1pt]

\node[spectral,xshift=0cm,yshift=0cm] (1) {$1$};
\draw[xshift=-0.3cm,yshift=0.4cm] node{1};

\node[spectral,xshift=0cm,yshift=-1.5cm] (3) {$1$};
\draw[xshift=-0.3cm,yshift=-1.1cm] node{3};

\node[pseudospectral,xshift=4cm,yshift=0cm] (2) {$1$};
\draw[xshift=3.6cm,yshift=0.4cm] node{2};

\node[subspectral1,xshift=1cm,yshift=-3cm] (6) {$0$};
\draw[xshift=1cm,yshift=-2.6cm] node{6};

\node[subspectral1,xshift=2cm,yshift=-1.5cm] (4) {$0$};
\draw[xshift=2cm,yshift=-1.1cm] node{4};

\node[subspectral2,xshift=4cm,yshift=-1.5cm] (5) {$0$};
\draw[xshift=3.6cm,yshift=-1.1cm] node{5};

\node[superspectral,xshift=3cm,yshift=-3cm] (7) {$2$};
\draw[xshift=3cm,yshift=-2.6cm] node{7};

\draw (3) to (1); \draw (6) to (3); \draw (6) to (4); \draw (4) to
(2); \draw (5) to (2); \draw (7) to (5);
\end{tikzpicture}
\end{center}
\caption{The marked (reduced) graph of matrix $A$. Circles correspond to the nodes with the
greatest Perron root $1$, the darker ones being spectral. Each node is marked by its Perron root (inside)
and by its number (outside).}
\end{figure}

Let $b \in  \R^7_+$. It follows from Theorem~\ref{Th main 2} (iv)
that there exists a solution $x$ to the max times equation $Ax+b =
x$  if and only if $\supp(b) \subseteq \{1,3,4,6\}$. In the usual
nonnegative algebra, the condition is more restrictive:
$\supp(b)\subseteq \{4,6\}$.

We choose
\[ b =
\begin{pmatrix}
0  & 0 & 0 & 1 & 0 & 0 & 0
\end{pmatrix}^T \]
as in \cite{HSHershkowitz} Example 3.22. Then  $\supp(b) = \{4\}$
and the minimal solution $x^0$ of $Ax+b =x$  has support $\{4,6\}$
and equals
\[x^0=
\begin{pmatrix}
0  & 0 & 0 & 1 & 0 & 1 & 0
\end{pmatrix}^T. \]
in both theories.

In max-times algebra, $\{1\}$ and $\{3\}$ are spectral nodes, and
the eigenvector cone for the eigenvalue $1$ is generated by
\[v^1=
\begin{pmatrix}
1& 0 & 1 & 0 & 0 & 1 & 0
\end{pmatrix}^T \]
and
\[v^2=
\begin{pmatrix}
0& 0 & 1 & 0 & 0 & 1 & 0
\end{pmatrix}^T, \]
see Theorems~\ref{spectral} and~\ref{t:red-spec}. In the usual
nonnegative algebra, $\{3\}$ is the only spectral node and any
eigenvector is a multiple of $v^2$.

In max-times algebra, the maximal support of a solution is
$\{1,3,4,6\}$. For example take
\[y^1 =
\begin{pmatrix}
2& 0 & 3 & 1 & 0 & 3 & 0
\end{pmatrix}^T, \]
the max-times "sum" of $x^0$, $2v^1$ and $3v^2$. In the usual
nonnegative algebra, the maximal support is $\{3,4,6\}$, take
\[y^2= x^0+v^2=
\begin{pmatrix}
0& 0 & 1 & 1 & 0 & 2 & 0
\end{pmatrix}^T,
\]
as in~\cite{HSHershkowitz}. Note that neither $y^1$ is a solution in
the usual sense, nor $y^2$ is a solution in the max-times sense.

Observe that for given $A,b$, if the usual Perron roots of all
blocks in FNF are the same as the max-times roots (as in the example
above), the existence of a solution in nonnegative algebra implies
the existence in max algebra (but not conversely). Examples of
vectors for which a solution exists with $A$ as above in max algebra
but not in nonnegative algebra are given by $v^1$ and $v^2$.

In the case of existence, minimal solutions have the same support in
both theories.

\section{Algebraic generalizations}
\label{s:algen}

Here we describe general setting in which
Theorems~\ref{t:equivalence} and~\ref{Th main 1} of
Subsection~\ref{ss:general} hold.

Recall that a set $\semr$ is called a {\em semiring} if it is
equipped with operations of addition $+$ and multiplication $\cdot$
satisfying the following laws:
\begin{enumerate}
\item[(a)] addition is commutative $a+b=b+a$ $\forall a,b\in\semr$,
\item[(b)] multiplication distributes over addition $a(b+c)=ab+ac$
$\forall a,b,c\in\semr$,
\item[(c)] both addition and multiplication
are associative: $a+(b+c)=(a+b)+c$, $a(bc)=(ab)c$ $\forall
a,b,c\in\semr$,
\item[(d)] there are elements $\bzero$ and $\bunity$
such that $a+\bzero=a$, $a\bunity=\bunity a=a$, and $a\bzero =\bzero
a=\bzero$ for all $a\in\semr$.
\end{enumerate}
The max-times algebra and the usual nonnegative algebra are
semirings (also note that any ring and any field is a semiring), see
also other examples below.

We consider a semiring $\semr$ endowed with a {\em partial order}
$\leq$, i.e., binary relation $\leq$ such that 1) $a\leq b$, $b\leq
c$ imply $a\leq c$, 2) $a\leq b$, $b\leq a$ imply $a=b$, 3) $a\leq
a$. In the case of idempotent addition ($a+a=a$), one defines a 
{\em canonical order} by $a\leq b\Leftrightarrow a+b=b$.

To model both max-times algebra and the usual nonnegative
algebra, we may assume that the following axioms are satisfied.

\begin{itemize}
\item[A1.] any countable ascending chain (i.e., linearly ordered subset)  in $\semr$ bounded from above has
supremum, and any countable descending chain has infimum;
\item[A2.] addition is nondecreasing: $a+b\geq a$ and $a+b\geq b$;
\item[A3.] both operations distribute over any supremum or infimum of any such chain in
$\semr$, i.e.
\begin{equation}
\label{e:distr}
\begin{split}
a+ \sup_{\mu} b_{\mu} = \sup_{\mu} (a+b_{\mu}),\qquad  &
a'\cdot \sup_{\mu} b_{\mu}= \sup_{\mu} (a'\cdot b_{\mu})\\
c+ \inf_{\mu} d_{\mu} = \inf_{\mu} (c+d_{\mu}),\qquad  & c'\cdot
\inf_{\mu} d_{\mu}= \inf_{\mu} (c'\cdot d_{\mu})
\end{split}
\end{equation}
for any countable bounded ascending chain
$\{b_{\mu}\}\subseteq\semr$, countable descending chain
$\{d_{\mu}\}\subseteq \semr$, elements $a,c,a',c'\in \semr$.
\end{itemize}

Axiom A2 implies that the semiring is {\bf nonnegative}: $a\geq 0$
for all $a$, and {\em antinegative}: $a+b=0$ implies $a=b=0$. Axiom
A3 implies that both arithmetic operations are monotone.

The operations of $\semr$ extend to matrices and vectors in the
usual way. Moreover we can compute matrix powers $A^k$ for $k\geq
0$, where we assume $A^0=I$, the identity matrix, where all diagonal
entries equal to $\bunity$ and all off-diagonal entries equal to
$\bzero$. The extension of notions of associated digraph and access
relations is also evident, provided that there are no zero divisors.

Note that partial order in $\semr$ is extended to $\semr^n$ and
$\semr^{m\times n}$ (matrices with $m$ rows and $n$ columns over
$\semr$) componentwise. The monotonicity of addition and
multiplication is preserved. Moreover,
distributivity~\eqref{e:distr} also extends to matrices and vectors:
\begin{equation}
\label{e:distmats}
\begin{split}
A+ \sup_{\mu} B^{\mu} = \sup_{\mu} (A+B^{\mu}),\qquad  &
A'\cdot \sup_{\mu} B^{\mu}= \sup_{\mu} (A'\cdot B^{\mu}),\\
C+ \inf_{\mu} D^{\mu} = \inf_{\mu} (C+D^{\mu}),\qquad & C'\cdot
\inf_{\mu} D^{\mu}= \inf_{\mu} (C'\cdot D^{\mu}).
\end{split}
\end{equation}
Here $\{B^{\mu}\},\{D^{\mu}\}$ are chains of matrices (ascending and
descending, respectively), where $\{B^{\mu}\}$ is bounded from
above.

Indeed, the distributivity of addition is verified componentwise.
Let us verify the $\sup$-distributivity for multiplication. Let $n$
be the number of columns of $C'$. We have:
\begin{equation}
\label{matrix-inf} (C'\cdot\sup_{\mu\in\Nat} D^{\mu})_{ik} =
\sum_{j=1}^n c'_{ij}(\sup_{\mu\in\Nat} d_{jk}^{\mu}) = \sup_{\kappa}
\sum_{j=1}^n c'_{ij} d_{jk}^{\kappa(j)},
\end{equation}
where $\Nat$ denotes a countable set and the last supremum is taken over all mappings $\kappa$ from
$\{1,\ldots,n\}$ to the natural numbers. The last equality is due to
the scalar $\sup$-distributivity. Now denote $\nu:=\max_{j=1}^n
\kappa(j)$ and observe that
\begin{equation}
\sum_{j=1}^n c'_{ij} d_{jk}^{\kappa(j)} \leq \sum_{j=1}^n c'_{ij}
d_{jk}^{\nu},
\end{equation}
since $d_{jk}^{\nu}$ are ascending chains. This implies that in the
last supremum of~\eqref{matrix-inf} we can restrict maps
$\kappa$ to identity, obtaining that
\begin{equation}
\sup_{\kappa} \sum_{j=1}^n c'_{ij} d_{jk}^{\kappa(j)} =\sup_{\nu}\sum_{j=1}^n c'_{ij} d_{jk}^{\nu} =\sup_{\nu} (C'D)_{ik}.
\end{equation}
Thus the matrix $\sup$-distributivity also holds. The
$\inf$-distributivity can be checked along the same lines replacing
infimum by supremum, and ascending chains by descending chains.

It can be checked that axioms A1-A3 and matrix
distributivity~\eqref{e:distmats} provide sufficient ground for the
proofs of Theorems~\ref{t:equivalence} and~\ref{Th main 1}.

For the alternative proof of Theorem~\ref{Th main 1} given in Remark~\ref{r:krivulin} the system A1-A3 has to be modified.
Note that the main part of the proof
after~\eqref{leastwdef} does not need anything but the existence of
sups of bounded ascending chains and the distributivity of addition over such sups. It is only the starting representation $x=u+\underline{w}$, where
$u=A^*b$ and $\underline{w}$ is the least vector $w$ satisfying $x=u+w$, which may need more than that.

We impose A1, A2 and the part of A3 asserting the distributivity of
addition and multiplication with respect to sup's of ascending chains, which is needed for Theorem~\ref{t:equivalence}, that is, for $u=A^*b$ to be the least solution of
$x=Ax+b$. We also require that there is at least one vector
$w$ satisfying $x=u+w$. This will hold if we impose
\begin{itemize}
\item[A4.] For all $a,c\in\semr$ such that $a\leq c$ there is $b\in\semr$ such that $a+b=c$.
\end{itemize}
In addition, in order to get the least $w$ satisfying $x=u+w$,
we impose the distributivity of $+$ with respect to
arbitrary inf's. Notice that in the case of an idempotent semiring
we define the order canonically  ($a\leq b\Leftrightarrow a+b=b$), 
and the axioms A2 and A4 are satisfied automatically.

Now we consider some examples of semirings where Theorems~\ref{t:equivalence} and~\ref{Th main 1} are valid.
In particular, axioms A1-A3 are satisfied for these examples.

{\bf Examples 1,2:} Max-times algebra, classical nonnegative algebra
(see Prerequisites).

{\bf Example 3: Max-min algebra.} Interval $[0,1]$ equipped with $a
b:=\min(a,b)$ and $a+b:=\max(a,b)$.

{\bf Example 4: {\L}ukasiewicz algebra.} Interval $[0,1]$ equipped
$ab:=\max(0,a+b-1)$ and $a+b:=\max(a,b)$.

{\bf Example 5: Distributive lattices.} Recall that a {\em lattice}
is a partially ordered set \cite{Bir:79} where any two elements
$a,b$ have the least upper bound $a\vee b:=\sup (a,b)$ and the
greatest lower bound $a\wedge b:=\inf(a,b)$. A lattice is called
{\em distributive} if the following laws hold:
\begin{equation}
\label{e:distrlat} a\vee (b\wedge c)=(a\vee b)\wedge (a\vee c),\quad
a\wedge (b\vee c)=(a\wedge b)\vee (a\wedge c)
\end{equation}
When a lattice also has the lowest element $\epsilon$ and the
greatest element $\top$, it can be turned into a semiring by setting
$ab:= a\wedge b$, $a+b:=a\vee b$, $\bzero=\epsilon$ and
$\bunity=\top$. To ensure that the axioms A1 and A3 hold, we require
that the lattice is {\em complete}, i.e., that $\vee_{\alpha}
a_{\alpha}$ and $\wedge_{\beta} b_{\beta}$ exist for all subsets
$\{a_{\alpha}\}$ and $\{b_{\beta}\}$ of the lattice, and that the
distributivity can be extended:
\begin{equation}
\label{e:infdistrlat} a\vee \wedge_{\beta}
b_{\beta}=\wedge_{\beta}(a\vee b_{\beta}),\quad b\wedge
\vee_{\alpha} a_{\alpha}= \vee_{\alpha} (b\wedge a_{\alpha}).
\end{equation}
Max-min algebra is a special case of this example.

{\bf Example 6: Idempotent interval analysis}. Suppose that $a_1,a_2\in\semr$ where $\semr$ satisfies
A1-A3, and consider the semiring of ordered pairs $(a_1,a_2)$, where $a_1\leq a_2$ and the operations of $\semr$
are extended componentwise. This semiring, satisfying A1-A3, is the basis of idempotent interval analysis as
introduced in~\cite{LS-98}.

{\bf Example 7: Extended order complete vector lattices}. We can consider
a semiring of all sequences $(a_1,a_2,\ldots)$ where $a_i \in\semr$ and $\semr$ satisfies A1-A3,
with the operations extended componentwise. A generalization of Example 6 is then a semiring of 
all ordered sequences $a_1\leq a_2\leq\ldots$ where $a_i\in\semr$.

{\bf Example 8: Semirings of functions.} Further extension of Example 7 to functions on a continuous domain is also evident (following~\cite{CGQS-05,LMS-01}).
As an example of a subsemiring of functions satisfying A1-A3, one may consider convex functions on
the real line, equipped with the operations of componentwise max (as addition) and componentwise
addition (as multiplication). In the spirit of max-plus semiring, we allow a function to take $-\infty$ values (which are absorbing). 
To verify A1-A3, recall that a function $f$ is convex if the set 
$\{(x,t)\mid t\geq f(x)\}$ is convex (providing connection to the well-known properties of convex sets).
In particular, the addition corresponds to the intersection of convex sets, and the multiplication corresponds to the Minkowski sum of convex sets. Note that the inf of descending chain of convex functions is also computed componentwise. 
As another example, we can consider a semiring, where an element is a class of functions on a continuous domain different only on a countable subset. Then, all countable sups or infs are well defined, since any two members
of the class corresponding to such sup or inf will differ only on a countable subset, and axioms A1-A4 are verified componentwise, as above.

\if{
{\bf Example 8: Conditionally complete lattice ordered group.} A
{lattice ordered group} is a group which is also a lattice. It is
called {\em conditionally complete} if $\vee_{\alpha} a_{\alpha}$
belongs to the group for any subset $\{a_{\alpha}\}$ bounded from
above. Dually it can be shown that $\wedge_{\beta} b_{\beta}$
belongs to the group for any subset $\{b_{\beta}\}$ bounded from
below. The lattice can be further completed by the least element
$\epsilon$, with arithmetic  $a\epsilon =\epsilon a=\epsilon$. Then
this can be converted to a semiring by setting $a+b:=a\vee b$,
$\bzero:=\epsilon$, $\bunity=e$ (the neutral element of the group),
and  keeping the group law as multiplication. Indeed, it can be
verified that the law $(a\vee b) c= ac\vee bc$ holds for all
elements $a,b,c$ of a lattice-ordered group. By a theorem of
Iwasawa, see \cite{Bir:79}, Theorem 28, conditionally complete
lattice ordered group is commutative.  It can be deduced
from~\cite{Bir:79}, Theorem 26 that
\begin{equation}
\label{inf-distr-mult} c\cdot\mysup_{\alpha} a_{\alpha}
=\mysup_{\alpha} c a_{\alpha},\quad c\cdot\myinf_{\beta} b_{\beta} =
\myinf_{\beta} c b_{\beta}
\end{equation}
for any element $c$, and any subsets $\{a_{\alpha}\}$ and
$\{b_{\beta}\}$ of the group. Thus the part of \eqref{e:distr} with
multiplication (right-hand side) is satisfied. Moreover,
by~\cite{Bir:79}, Theorem 25 we also have that the lattice of any
lattice-ordered group is distributive, and the laws
\begin{equation}
\label{inf-distr-add} c\wedge\mysup_{\alpha} a_{\alpha}
=\mysup_{\alpha} (c\wedge a_{\alpha}),\quad c\vee \myinf_{\beta}
b_{\beta} = \myinf_{\beta} (c\vee b_{\beta})
\end{equation}
hold whenever $\mysup_{\alpha} a_{\alpha}$ and $\myinf_{\beta}
b_{\beta}$ belong to the group. This implies that the part of
\eqref{e:distr} with addition is satisfied, if we also impose that
$\myinf_{\beta} (c\vee b_{\beta})=c$ when $\{b_{\beta}\}$ is not
bounded from below. Thus, a lattice-ordered group with this
additional assumption provides another example of a semiring
satisfying axioms (A1) - (A3). Note that max-times algebra is a
special case of this example.
}\fi

\medskip
Note that the Kleene star~(\ref{closures}) always converges in
Examples 3-5. Moreover, it can be truncated for $k\geq n$, for an
$n\times n$ matrix, so that $A^*=I+A+\ldots + A^{n-1}$, which
follows from the optimal path interpretation, the finiteness of
associated digraph, and because the matrix entries do not exceed
$\bunity$. Hence $A^*b$ is well defined for any $A,b$, and $x=Ax+b$
is always solvable, with the solution set described by
Theorem~\ref{Th main 1}. In Examples 6-8 the convergence of Kleene star should hold
for the matrices corresponding to each component of the semiring (for the last ``subexample'', excluding a 
countable subset of components).

Finally we observe that theorems formally like \ref{t:equivalence}
and \ref{Th main 1} also hold in the case a linear operator leaving
invariant a proper cone in ${\mathbb R}^n$, see \cite{TS-03}, Theorem
3.1, where an analog of Theorem \ref{Th main 2} is also proved.

\if{
\subsection{Case of distributive lattice}

In the case of distributive lattice the theory of $Z$-equations
$\lambda x = Ax+b$ is different from the theory of $x=Ax+b$ since
$\lambda$ is not invertible. However, we observe the following.

\begin{proposition}
\label{zeq-distlat} Over a distributive lattice, equation~$\lambda
x=Ax+b$ is solvable if and only if $b_i\leq\lambda$ for all $i$. In
this case,
\begin{itemize}
\item[a)] $x^0=A^*b$ is the least solution of $\lambda x=Ax+b$ ;
\item[b)] For any solution $x$ of $\lambda x= Ax+b$, $\lambda x$ is a solution
of $x=Ax+b$. Hence $\lambda x= x^0+v$ where $v$ satisfies $Av=v$.
\end{itemize}
\end{proposition}
\begin{proof}
Clearly, for the solvability of $\lambda x = Ax+b$ it is necessary that
$b_i\leq\lambda$ for all $\lambda$. In this case,the same holds for $x^0=A^*b$.
As $x^0$ satisfies $x=Ax+b$, it also satisfies
$\lambda x= Ax+b$, hence the condition is necessary and sufficient.
Now assume that it holds. \\
a): If there is another solution $y$, then taking $\wedge$ of the
two Z-equations and exploiting the monotonicity of $A$ we conclude
that $y\wedge x^0$ satisfies $\lambda (y\wedge x^0)\geq A(y\wedge
x^0) +b$ and hence $(y\wedge x^0)\geq A(y\wedge x^0) +b$. Iterating
this inequality and using $x^0=A^*b$ we obtain
$y\wedge x^0\geq x^0$.\\
b): Multiply~\eqref{zeq} by $\lambda$ using
$\lambda\lambda=\lambda$, and apply Theorem~\ref{Th main 1}.
\end{proof}

Now we restrict to $x=Ax+b$ (in some sense this can be justified by
Theorem~\ref{zeq-distlat}.) We aim to obtain an explicit formula for
$v=\inf_k A^k x$, in the representation $x=x^0+v$. Our argument is a
simple and more general version of the results in \cite{Cec-92}, and
similar to~\cite{Tan-98}. Denote by $A_x$ the matrix with columns
$(A_x)_{\cdot i}:=A_{\cdot i} x_i$, and by $\boldunity$ a vector
with all components equal to $\bunity$.

\begin{lemma}
\label{A_x} If $x$ satisfies $Ax\leq x$, then for all $k$ one has
$A^k x= A_x^k \boldunity$, over a distributive lattice.
\end{lemma}
\begin{proof}
We need to demonstrate that
\begin{equation}
\label{demo1} \sum_{i_1,\ldots, i_k} a_{i_0i_1}a_{i_1i_2}\ldots
a_{i_{k-1}i_k}x_{i_k} = \sum_{i_1,\ldots, i_k}
a_{i_0i_1}x_{i_1}a_{i_1i_2}x_{i_2}\ldots a_{i_{k-1}i_k}x_{i_k}.
\end{equation}
Indeed, using the inequality $a_{i_1i_2}x_{i_2}\leq x_{i_1}$ we can
suppress $x_{i_1}$ on the r.h.s., and we can proceed further until
$a_{i_{k-1}i_k}x_{i_k}\leq x_{i_{k-1}}$ which supresses
$x_{i_{k-1}}$ (but not $x_{i_k}$). So we obtain the l.h.s.
of~\eqref{demo1}.
\end{proof}

The following result is known, the proof is given for convenience of
the reader.

\begin{lemma}[\cite{Tan-98}]
\label{stabilize} Over a distributive lattice $L$, for $A\in
L^{n\times n}$ and $x$ satisfying $Ax\leq x$, vectors $A^kx$
stabilize starting from $k=n$.
\end{lemma}
\begin{proof}
By  Lemma~\ref{A_x} we can assume $x=\boldunity$. Then the entries
$(A^kx)_i$ represent the largest weights (capacities) of paths
starting at $i$. For $k\geq n$ any such path contains a cycle. Fix
$k$ and find an $l$ which divides the lengths of all elementary
cycles contained in such paths (there is only a finite number of
such paths and cycles). Then for each path with length $k$ starting
at $i$ there is a path of length $k+l$ starting at $i$ with the same
weight. Indeed, it suffices to append to the former path several
copies of any cycle, which does not change the weight (capacity) of
the path, by the idempotency of multiplication. Therefore
$(A^{k+l}x)_i\geq (A_k x)_i$. However, we also have
$(A^{k+l}x)_i\leq (A_k x)_i$ since $x$ satisfies $Ax\leq x$. This
shows $A^kx=A^{k+l}x$ and also $A^k x=A^{k+1}x=\ldots=A^{k+l}x$
since $A^k x\geq A^{k+1}x\geq\ldots\geq A^{k+l}x$.
\end{proof}

Note that as $x\leq\boldunity$ for any vector $x$, and as $x\leq y$
implies $Ax\leq Ay$, we obtain that $v\leq A^n\boldunity$, for any
vector $v$ satisfying $Av=v$. Thus, there exists the greatest
eigenvector of $A$ in any distributive lattice (with the greatest
element $\bunity$), generalizing \cite{Cec-92} Theorem 3.

Denote by $\digr(A,x)$ the weighted digraph of the matrix $A_x$
introduced above, and denote by $\digr(A,x,t)$ the subgraph
consisting of edges $(i,j)$ such that $a_{ij}\geq t$.

\begin{theorem}
\label{distlat} Over a distributive lattice $L$, suppose that $A\in
L^{n\times n}$ and $x$ satisfies $Ax\leq x$. Then the $i$th
component of $v=\inf_k A^k x$ equals
\begin{itemize}
\item[a)] The supremum of weights of all paths in $\digr(A,x)$ starting at $i$, which contain a cycle,
\item[b)] The supremum $t^*_i$ of levels $t$ such that node $i$ accesses a cycle in $\digr(A,x,t)$,
\end{itemize}
for any $i=1,\ldots,n$.
\end{theorem}
\begin{proof} By Lemma~\ref{A_x} we can assume w.l.o.g. that $x=\boldunity$.

a): By Lemma~\ref{stabilize} the sequence $A^k x$ stabilizes after
$k\geq n$, so $v_i$ is the greatest weight of all paths of length
$k\geq n$ that start at $i$. All such paths contain a cycle. But if
we have a path of smaller length containing a cycle, we can add more
copies of that cycle until the length is at least $n$, without
changing the weight. This shows a).

b): If we have a path with weight $t$ in $\digr(A,x)$ starting at
$i$ and containing a cycle, then this cycle will be accessed from
$i$ in $\digr(A,x,t)$ (by that path). This shows $t^*_i\geq v_i$.
The other way around, if a cycle is accessed from $i$ in
$\digr(A,x,t)$ then the weight of accessing path including the cycle
is greater than equal to $t$, hence $v_i\geq t^*_i$.

\end{proof}
}\fi

\if{

}\fi


\end{document}